\documentclass{article}
\usepackage{amssymb,amsmath,xcolor,color,framed,graphicx,amsthm,tikz}
\definecolor{shadecolor}{rgb}{0.95, 0.95, 0.86}
\newcommand{\C}{{\mathbb C}}

\newcommand{\supp}{\mathop \mathrm{supp}}

\newcommand{\semmi}[1]{}

\newcommand{\be}{\begin{equation}}
\newcommand{\ee}{\end{equation}}

\newtheorem{theorem}{Theorem}
\newtheorem{definition}{Definition}
\newtheorem{proposition}{Proposition}
\newtheorem{lemma}{Lemma}
\newtheorem{corollary}{Corollary}
\newtheorem*{remark}{Remark}
\renewenvironment{proof}[1][Proof]{\noindent {\bf #1.}\ }{\hfill {\bf Q.E.D.}}

\title{Equilibrium measures for a class of potentials with discrete rotational symmetries}
\author{F. Balogh, D. Merzi\\
\\
SISSA, via Bonomea 265, Trieste, Italy}

\begin{document}
\maketitle

\begin{abstract}
In this note the logarithmic energy problem with external potential $|z|^{2n}+tz^d+\bar{t}\bar{z}^d$ is considered in the complex plane, where $n$ and $d$ are positive integers satisfying $d\leq 2n$.
Exploiting the discrete rotational invariance of the potential, a simple symmetry reduction procedure is used to calculate the equilibrium measure for all admissible values of $n,d$ and $t$.

It is shown that, for fixed $n$ and $d$, there is a critical value $|t|=t_{cr}$ such that the support of the equilibrium measure is simply connected for $|t|<t_{cr}$ and has $d$ connected components for $|t|>t_{cr}$.
\end{abstract}


\section{Introduction}
The logarithmic energy problem with external potential $V(z)$ in the complex plane amounts to finding the minimizer of the \emph{electrostatic energy functional}
\be
{\mathcal I}_V(\mu) = \int\!\!\!\int \log\frac{1}{|z-w|}d\mu(z)d\mu(w)+\int V(z)d\mu(z)
\ee
in the space of probability measures in $\C$.
 Following \cite{Saff_Totik}, we say that  $V\colon \C \to (-\infty,\infty]$ is an \emph{admissible potential} if it is lower semi-continuous, the set $\{z\colon V(z) < \infty\}$ is of positive capacity and 
\be
V(z)-\log|z| \to \infty\quad  \mbox{ as } z\to \infty\ .
\ee
As shown in \cite{Saff_Totik}, for an admissible potential $V$ there exists a unique measure $\mu=\mu_V$ that minimizes the functional ${\mathcal I}_V$, referred to as the \emph{equilibrium measure} corresponding to $V$. Moreover, $\mu_V$ is the unique compactly supported measure of finite logarithmic energy for which the variational conditions
\begin{align}
\label{eq:var_eq}
V(z)+2\int \log\frac{1}{|z-w|}d\mu(w) &= F \quad  z \in \supp(\mu)\quad \mbox{quasi-everywhere}\\
\label{eq:var_ineq}
V(z)+2\int \log\frac{1}{|z-w|}d\mu(w) &\geq F \quad z \in \C \quad \mbox{quasi-everywhere}
\end{align}
are valid for some constant $F$ (a property is said to hold \emph{quasi-everywhere} on a set $S$ if it is valid at all points of $S$ minus a set of zero logarithmic capacity).

The equilibrium measure plays a fundamental role in the asymptotic questions of weighted approximation theory \cite{Saff_Totik}; it also appears naturally in the context of hermitian and normal matrix models with unitary invariant probability measures (see \cite{KKMWZ, Elbau_Felder, TBAZW, Elbau, Etingof_Ma, Takhtajan_Its, Bleher_Kuijlaars, BBLM} and references therein). 

The aim of this paper is to find the equilibrium measure for the potentials of the special form
\be
\label{ex_pot}
V(z) = \frac{1}{T}\left(|z|^{2n}-tz^d-\bar{t}\bar{z}^{d}\right)\ ,
\ee
where $n$ and $d$ are positive integers with $2n \geq d$, $T >0$ and 
\be
\left\{
\begin{array}{ll}
t \in \C & \mbox{ if }d < 2n\\
|t| < 1/2 &\mbox{ if } d=2n\ .
\end{array}
\right.
\ee
It is easy to see that these potentials are admissible. Moreover, these functions are invariant under the group of discrete rotations of order $d$:
\be
V\left(e^{\frac{2\pi i k}{d}}z\right) = V(z) \qquad k=0,\dots, d-1\ .
\ee
It is natural to expect that the corresponding equilibrium measure is invariant under the same group of symmetries, which motivates the following construction.
\begin{definition}
For a fixed positive integer $d$ and a Borel probability measure $\mu$ the associated \emph{$d$-fold rotated measure} $\mu^{(d)}$ is defined to be
\be
\mu^{(d)}  = \frac{1}{d}\sum_{k=0}^{d-1}\mu^{(d)}_k\ ,
\ee
where the $k$th summand is given by
\be
\mu^{(d)}_k(B) = \mu\left( \varphi_k^{-1}(B \cap S_k)\right)
\ee
for any Borel set $B\subseteq \C$, where
\be
S_k = \left\{z \in {\mathbb C}\  \colon \frac{2\pi k}{d}\leq \arg(z) < \frac{2\pi (k+1)}{d} \right\} \quad k = 0,\dots, d-1
\ee
and
\be
\varphi_k \colon \C \to S_k\ , \quad \varphi_k(re^{i\theta}) = r^{\frac{1}{d}}e^{\frac{i\theta}{d}}e^{\frac{2\pi i k}{d}}\ , \quad k = 0,\dots, d-1\ .
\ee
\end{definition}

\begin{remark}
Note that $z\in \supp(\mu^{(d)})$ iff $z^d \in \supp(\mu)$.
\end{remark}

\begin{center}
\includegraphics{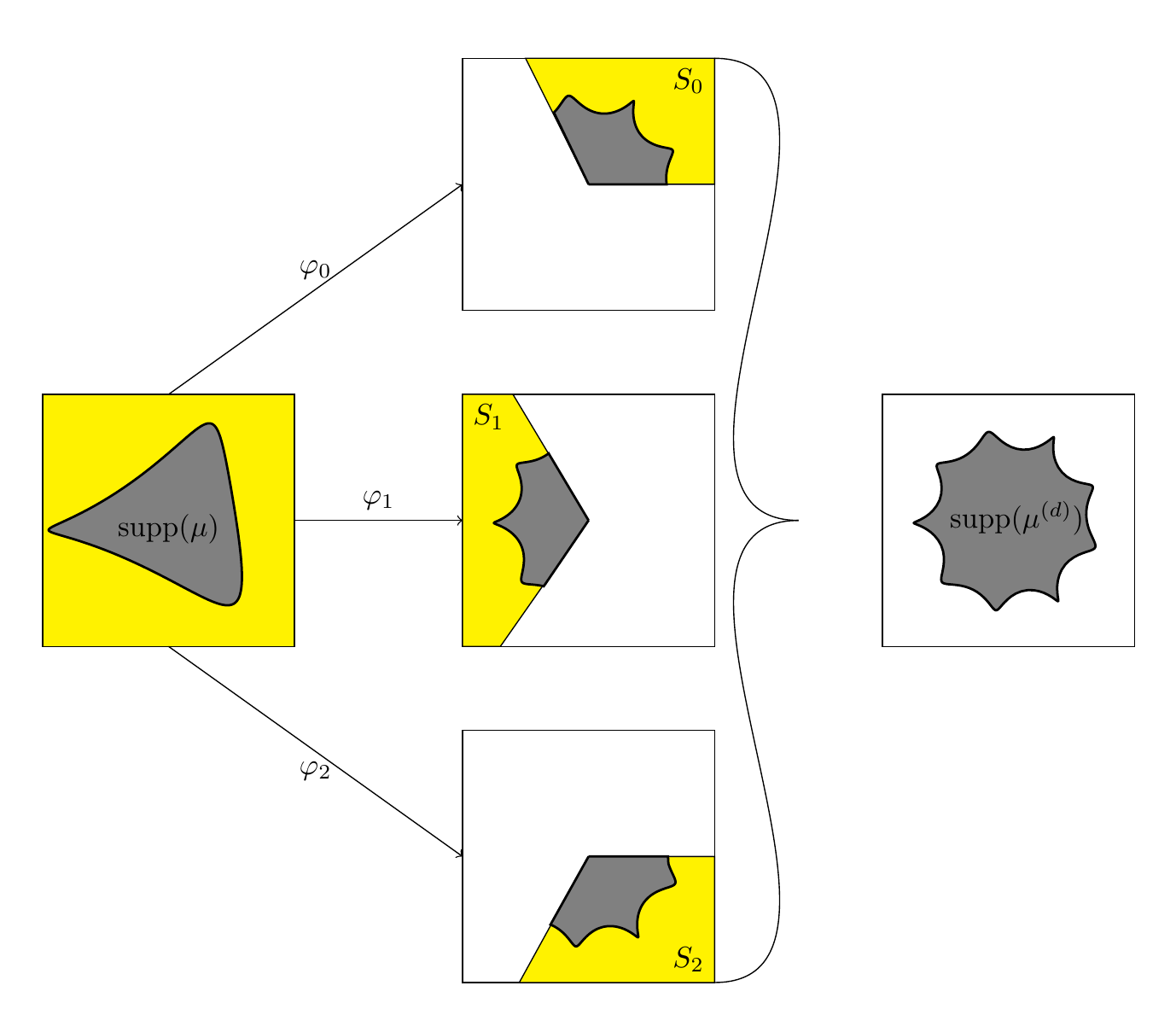}\\
{\bf Figure 1.} Illustration of the support of $\mu^{(d)}$ for $d=3$
\end{center}

\begin{shaded}
\begin{lemma}
\label{lemma:z_d_eq}
If the admissible potential $V(z)$ can be written in terms of another admissible potential $Q(z)$ as
\be
V(z) = \frac{1}{d}Q(z^d)
\ee
for some positive integer $d$, then the equilibrium measure for $V$ is given by the $d$-fold rotated equilibrium measure of $Q$, i.e., 
\be
\mu_V = \mu_Q^{(d)}\ .
\ee
\end{lemma}
\end{shaded}
The proof of this simple fact is given in Sec.~\ref{sec:z_d}. As a corollary, the problem of finding the equilibrium measure for the class \eqref{ex_pot} reduces to the study of the following family of potentials:
\be
\label{ex_pot_reduced}
Q(z) = \frac{d}{T}\left(|z|^{\frac{2n}{d}}-tz-\bar{t}\bar{z}\right)\ .
\ee
By differentiating \eqref{eq:var_eq} it is easy to find the density of $\mu_Q$ with respect to the Lebesgue measure $dA$ in the complex plane as
\be
d\mu_Q(z)= \frac{1}{4\pi}\Delta Q(z)\chi_K(z)dA(z)=\frac{n^2}{\pi Td}|z|^{\frac{2n}{d}-2}\chi_K(z)dA(z)\ ,
\ee
where $K=\supp(\mu_Q)$ and $\chi_K$ stands for the characteristic function of $K$.

The support set $K$ will be expressed in terms of its exterior uniformizing map
\be
f\colon \{u\in \C \colon\ |u| >1\} \to \C\setminus K
\ee
analytic and univalent for $|u|>1$ and fixed by the asymptotic behaviour
\be
\label{eq:conformal_normalized}
f(\infty)=\infty\ , \quad f(u) = ru\left(1+{\mathcal O}\left(\frac{1}{u}\right)\right)\quad u \to \infty\ ,
\ee
where $r$ is real and positive, called the \emph{conformal radius} of $K$.

\begin{center}
\includegraphics{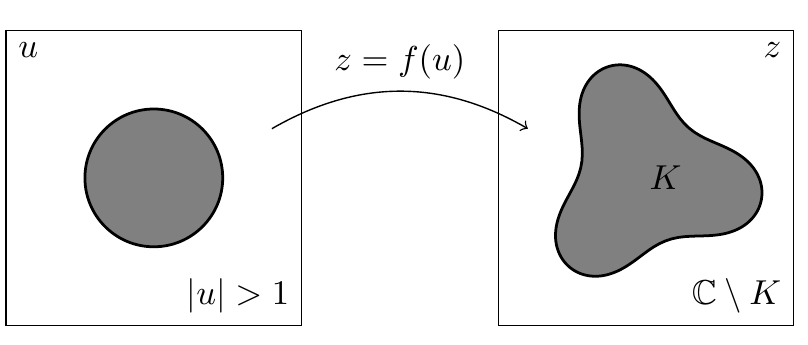}\\
{\bf Figure 2.} Illustration of the exterior conformal map
\end{center}
\begin{remark}
If the support $K$ of the measure $\mu$ contains $z=0$ and it is given by exterior uniformizing map 
\be
f(u)=rug(u) 
\ee
with
\be
g(u)\not =0 \quad |u|>1\ , \qquad g(u)=1+\mathcal{O}\left(\frac{1}{u}\right) \quad u \to \infty
\ee
 then the support of $\mu^{(d)}$ has the conformal map
\be
u \mapsto f(u^d)^{\frac{1}{d}} =r^{\frac{1}{d}}ug(u^d)^{\frac{1}{d}}
\ee 
which is well-defined, univalent in $|u|>1$ and normalized according to \eqref{eq:conformal_normalized}.
\end{remark}

\begin{remark}
The solution for the special cases $d=n$ and $d=2n$ are well-known.

For $d=n$ the reduced potential is
\be
\label{circ_pot}
Q(z) = \frac{n}{T}\left(|z|^2-tz-\bar{t}\bar{z}\right)=\frac{n}{T}|z-\bar{t}|^2 -\frac{n}{T}|t|^2\ ,
\ee
and therefore the equilibrium measure is the normalized area measure on the disk
\be
K=\left\{z\colon\ |z-\bar{t}| \leq \sqrt{\frac{T}{n}}\right\}\ ,
\ee
with exterior conformal map
\be
f(u)=ru+\bar{t} = ru\left(1+\frac{\bar{t}}{r}\frac{1}{u}\right)\ , \qquad r= \sqrt{\frac{T}{n}}\ .
\ee
For $d=2n$ we have
\be
Q(z) = \frac{2n}{T}\left(|z|-tz-\bar{t}\bar{z}\right)\ ,
\ee
which is just a re-parametrization of the well-known case
\be
\tilde Q(z) =\frac{1}{2}Q(z^2)= \frac{n}{T}\left(|z|^2-tz^2-\bar{t}\bar{z}^2\right)\ .
\ee
The equilibrium measure for $\tilde Q$ is the normalized area measure on an ellipse with exterior conformal map \cite{DFGIL, TBAZW, Elbau_Felder}
\be
\tilde f(u) = \sqrt{r}u\left(1+\frac{2\bar{t}}{u^2}\right)\ , \qquad r=\frac{T}{n}\frac{1}{1-4|t|^2}\ .
\ee
By applying Lemma \ref{lemma:z_d_eq} for $\tilde Q \to Q$ we get that the uniformizing map for the support of $\mu_Q$ is
\be
f(u) = \tilde f(\sqrt{u})^{2}=ru\left(1+\frac{2\bar{t}}{u}\right)^{2}\ .
\ee
Given the admissibility condition $|t|<\frac{1}{2}$ it is easy to see that $f(u)$ is univalent in $|u|>1$.
\end{remark}

Our main result is summarized by the following theorem and its immediate corollary.
\begin{shaded}
\begin{theorem} 
\label{thm:main} Let $0<d<n$ or $n<d<2n$, and define the critical value of the parameter $t$ as
\be
t_{cr}=\frac{n}{d}\left(\frac{T}{2n-d}\right)^{\frac{2n-d}{2n}}\ .
\ee

The equilibrium measure for the potential
\be
Q(z)=\frac{d}{T}\left(|z|^{\frac{2n}{d}}-tz-\bar{t}\bar{z}\right)
\ee
is absolutely continuous with respect to the area measure with density
\be
d\mu_Q(z) = \frac{n^2}{\pi Td}|z|^{\frac{2n}{d}-2}\chi_K(z)dA(z)\ ,
\ee
and its support $K$ is simply connected. The exterior uniformizing map of $K$ is of the form
\be
\label{eq:conformal_map}
f(u) = \left\{
\begin{array}{ll}
\displaystyle ru\left(1-\frac{\alpha}{u}\right)^{\frac{d}{n}} & |t| \leq t_{cr}\\
&\\
\displaystyle r\left(u-\frac{1}{\bar{\alpha}}\right)\left(1-\frac{\alpha}{u}\right)^{\frac{d}{n}-1} & |t| > t_{cr}\ ,
\end{array}
\right.
\ee
where $r=r(t)$ and $\alpha=\alpha(t)$ are given as follows:
\begin{itemize}
\item {\bf Pre-critical case $0<|t|< t_{cr}$ :}\\
The radius $r=r(t)$ is a particular solution of the equation
\be
\label{eq:r}
r^{\frac{4n}{d}-2}-\frac{T}{n}r^{\frac{2n}{d}-2}+\frac{n-d}{n}\frac{d^2}{n^2}|t|^2=0\ .
\ee
\begin{enumerate}
\item For $0< d <n$, \eqref{eq:r} has two distinct positive solutions $0 < r_{-}(t) < r_{+}(t)< \left(\frac{T}{n}\right)^{\frac{d}{2n}}$ such that $r_{-}(0)=0$ and $r_{+}(0)= \left(\frac{T}{n}\right)^{\frac{d}{2n}}$ and $r_{-}(t_{cr})=r_{+}(t_{cr})$. The conformal map corresponds to the choice  $r=r_{+}(t)$.
\item For $n <d<2n$,  \eqref{eq:r} has a unique positive solution $r=r(t)$.
\end{enumerate}
The value of the parameter $\alpha$ is given in terms of $r$ as
\be
\alpha = -\frac{d}{n}\bar{t}r^{-\frac{2n-d}{d}}\ .
\ee
\item {\bf Critical case $|t|= t_{cr}$ :}\\
The parameters $r$ and $\alpha$ are given explicitly by
\be
r=r_{cr}= \left(\frac{T}{2n-d}\right)^{\frac{d}{2n}}\ ,
\ee
and
\be
\alpha=\alpha_{cr} =-\frac{\bar{t}}{|t|}\ .
\ee
\item {\bf Post-critical case $|t|> t_{cr}$ :}\\
The parameters $r$ and $\alpha$ are given explicitly by
\be
r= \left(\frac{T}{2n-d}\right)^{\frac{1}{2}}\left(\frac{d}{n}|t|\right)^{\frac{d-n}{2n-d}}\ ,
\ee
and
\be
\alpha  =-\frac{\bar{t}}{|t|}\left(\frac{T}{2n-d}\right)^{\frac{1}{2}}\left(\frac{d|t|}{n}\right)^{-\frac{n}{2n-d}}\ .
\ee
\end{itemize}
\end{theorem}
\end{shaded}
\newpage
\newcommand{\tvalue}[1]{\raisebox{3\height}{\parbox{2cm}{\begin{center}$t=t_{cr}#1$\end{center}}}}
\begin{center}
\includegraphics[width=0.24\textwidth]{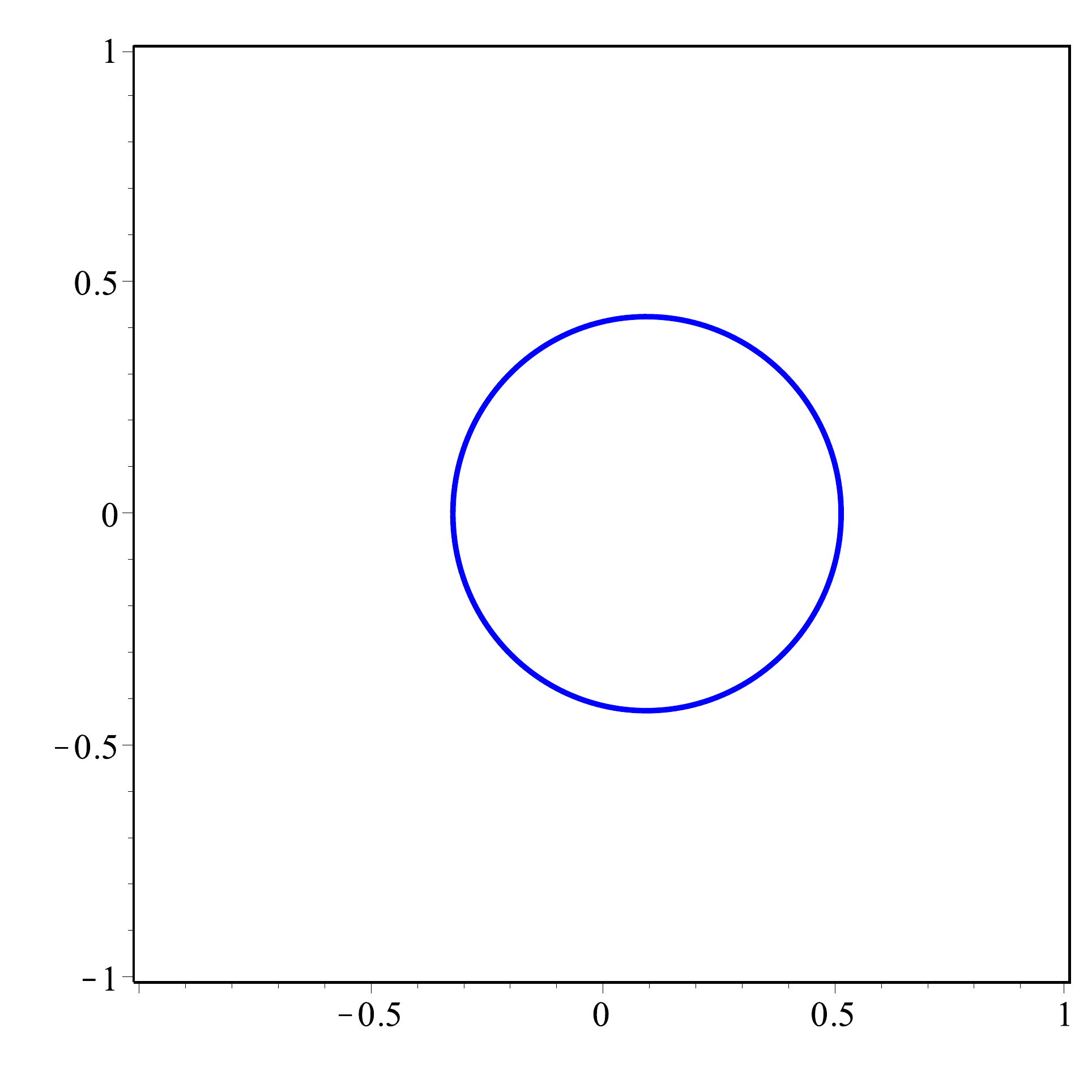}\tvalue{-0.2}\includegraphics[width=0.24\textwidth]{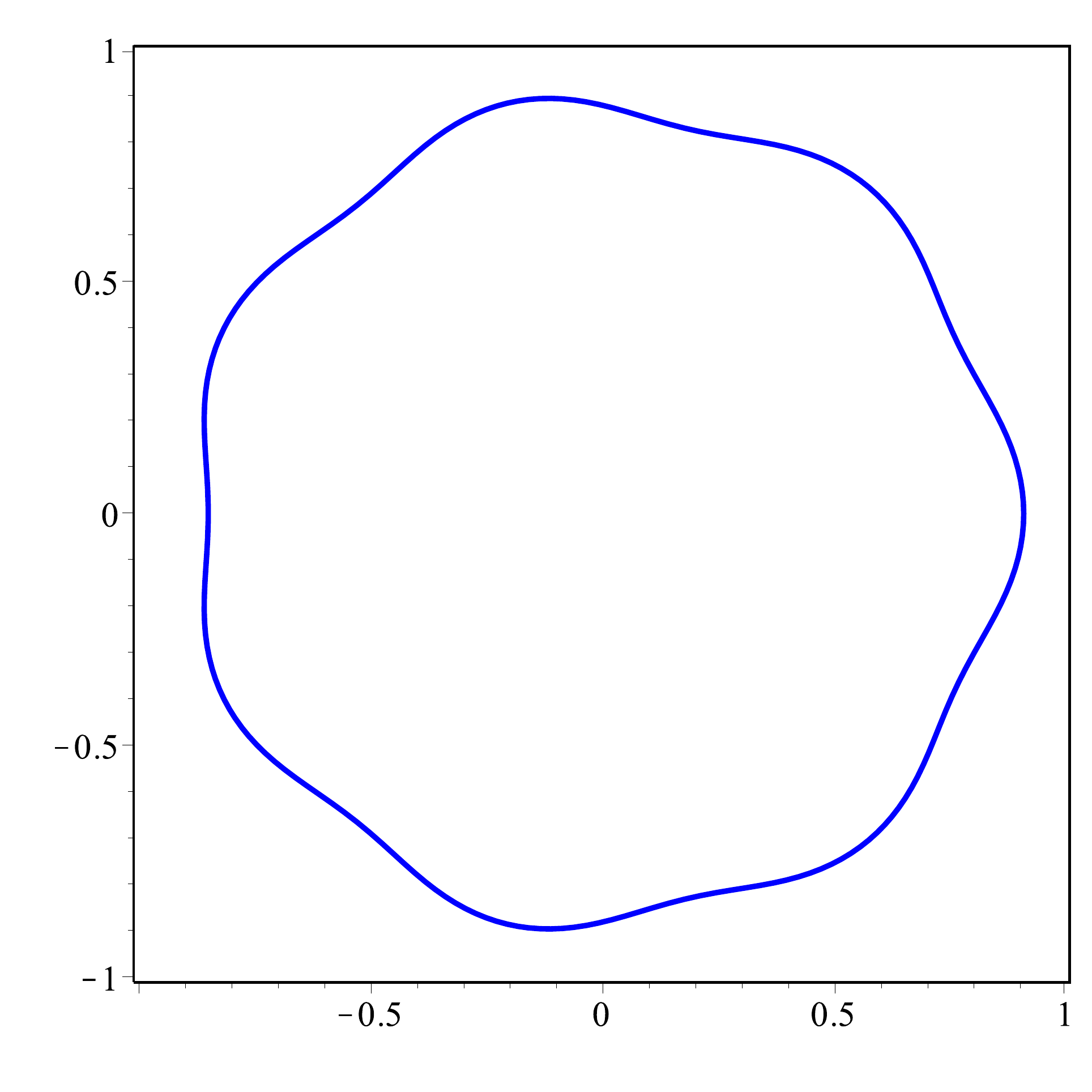}\\
\includegraphics[width=0.24\textwidth]{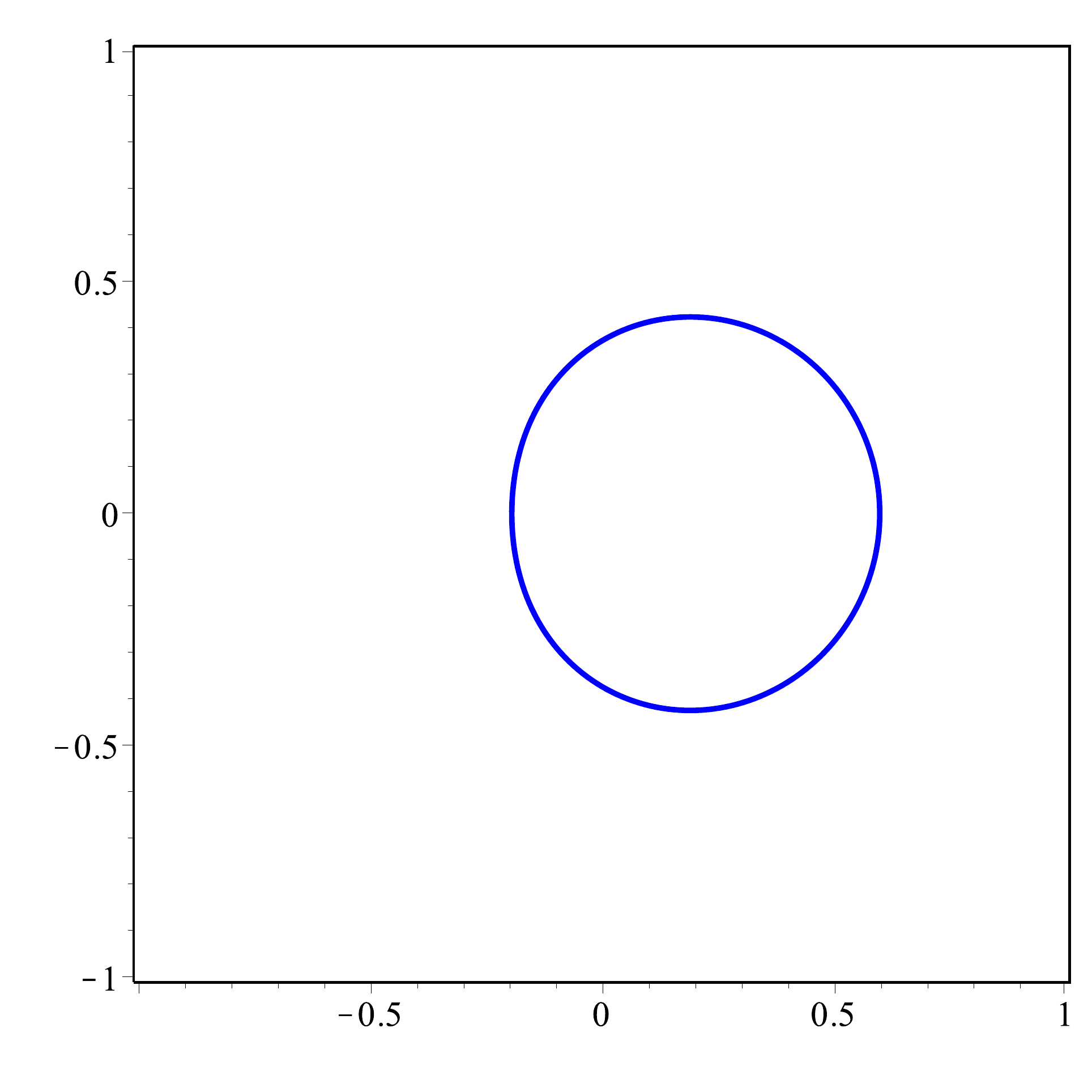}\tvalue{-0.1}\includegraphics[width=0.24\textwidth]{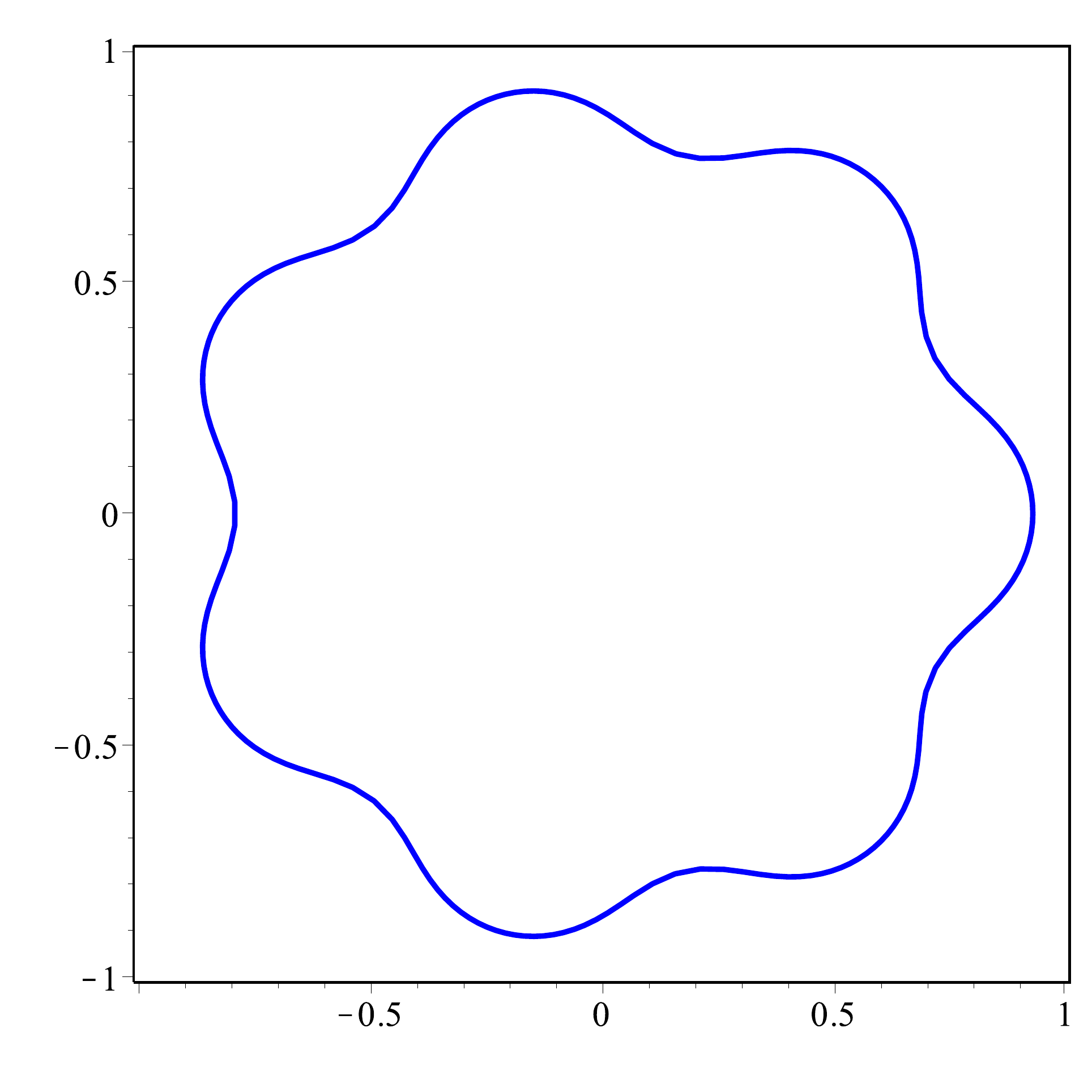}\\
\includegraphics[width=0.24\textwidth]{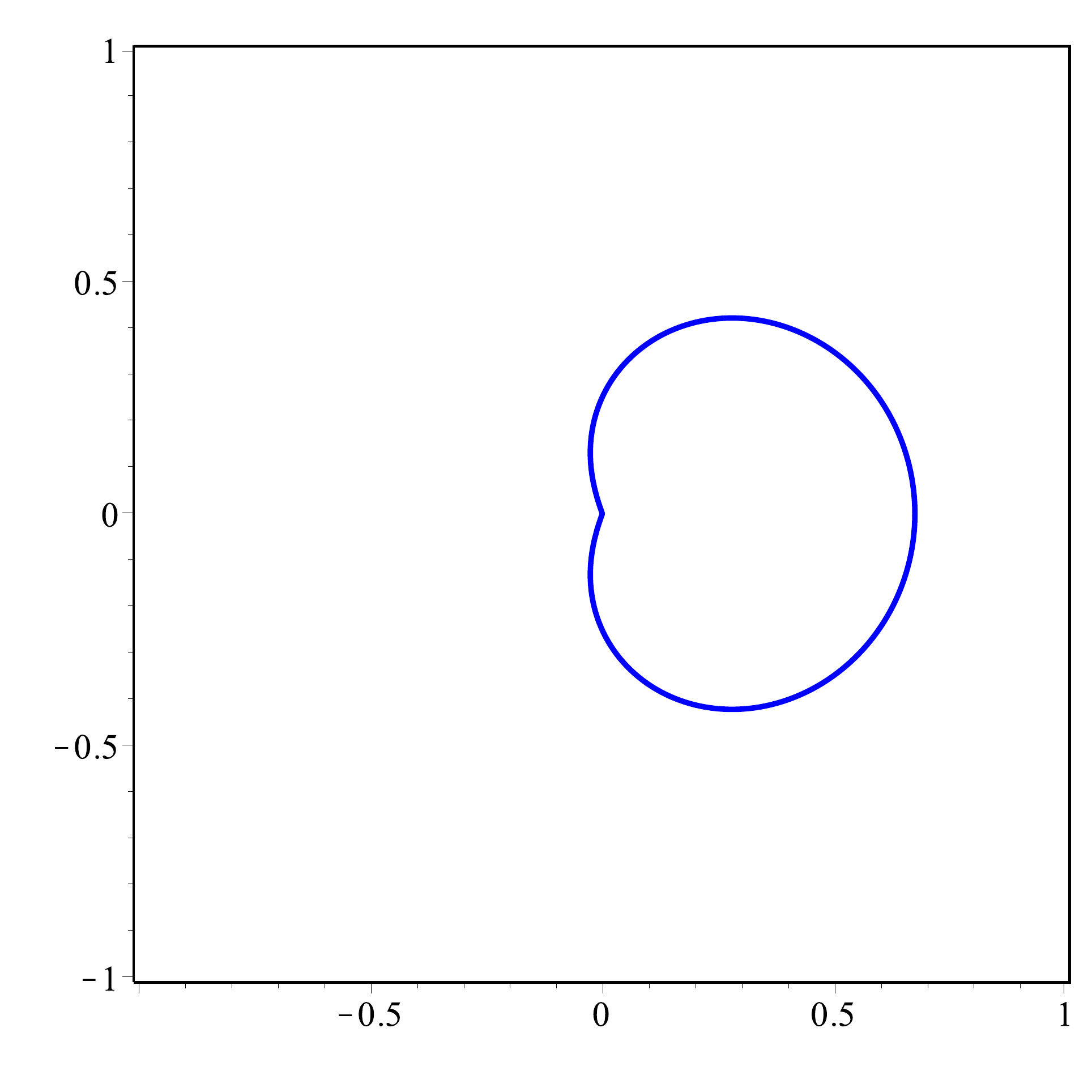}\tvalue{}\includegraphics[width=0.24\textwidth]{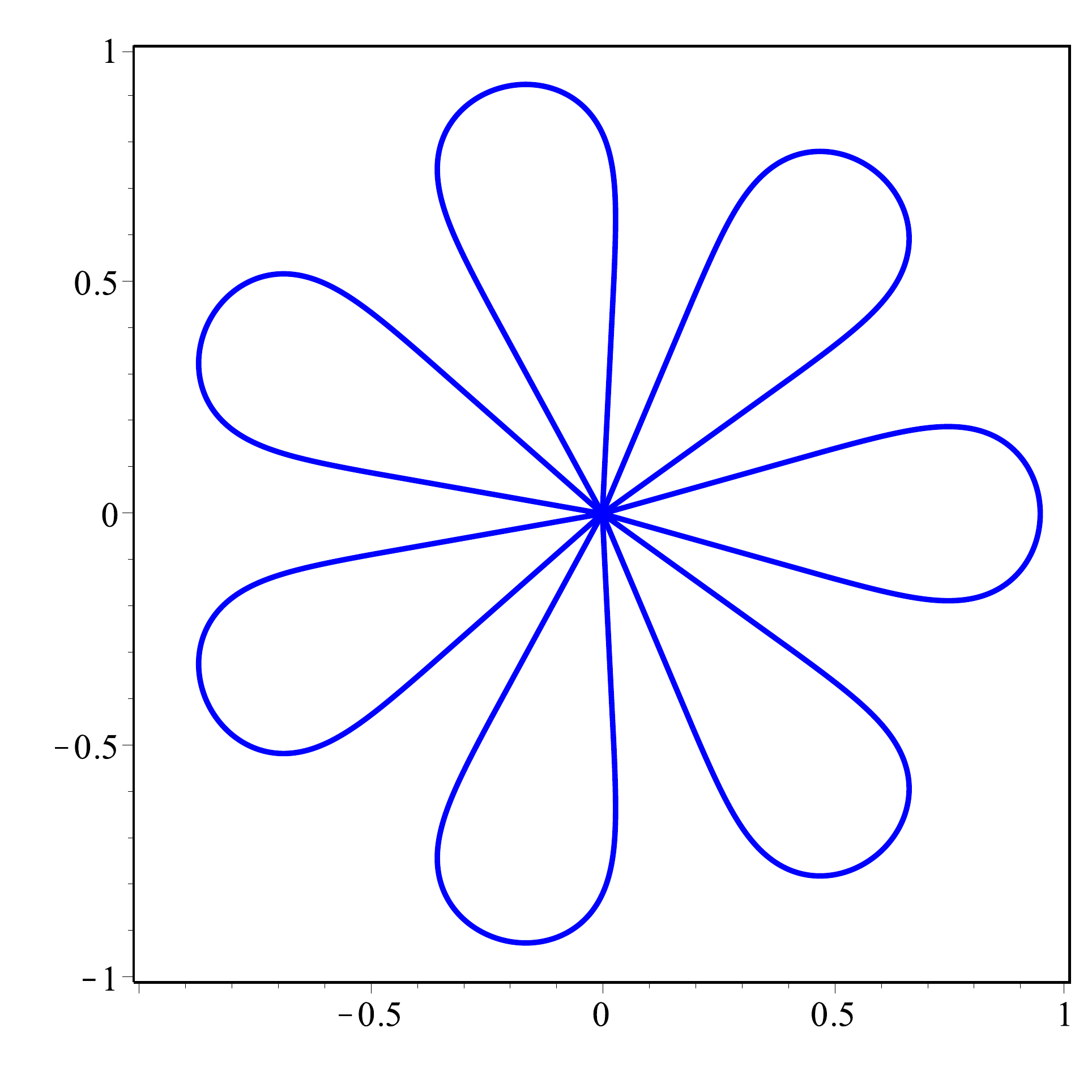}\\
\includegraphics[width=0.24\textwidth]{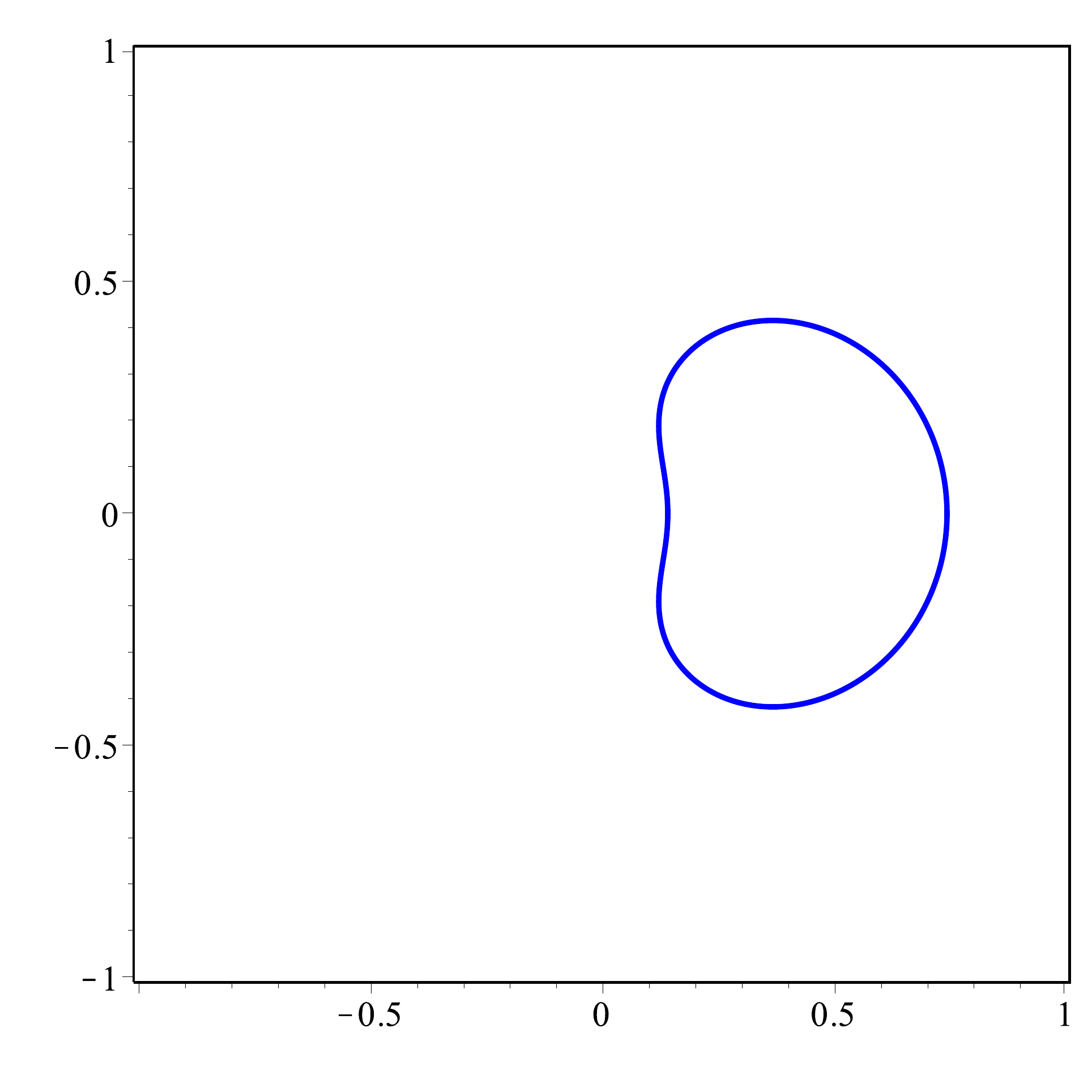}\tvalue{+0.1}\includegraphics[width=0.24\textwidth]{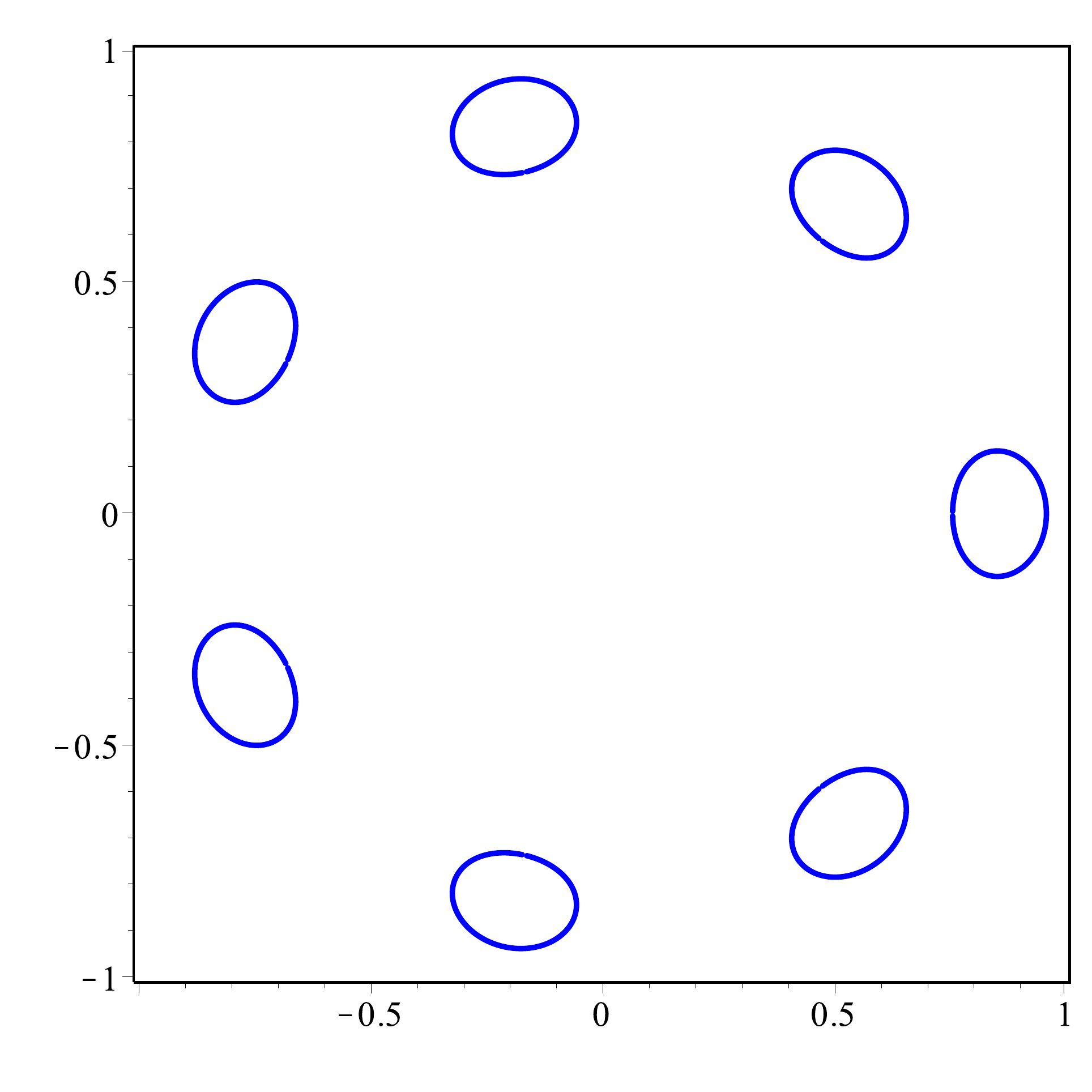}\\
\includegraphics[width=0.24\textwidth]{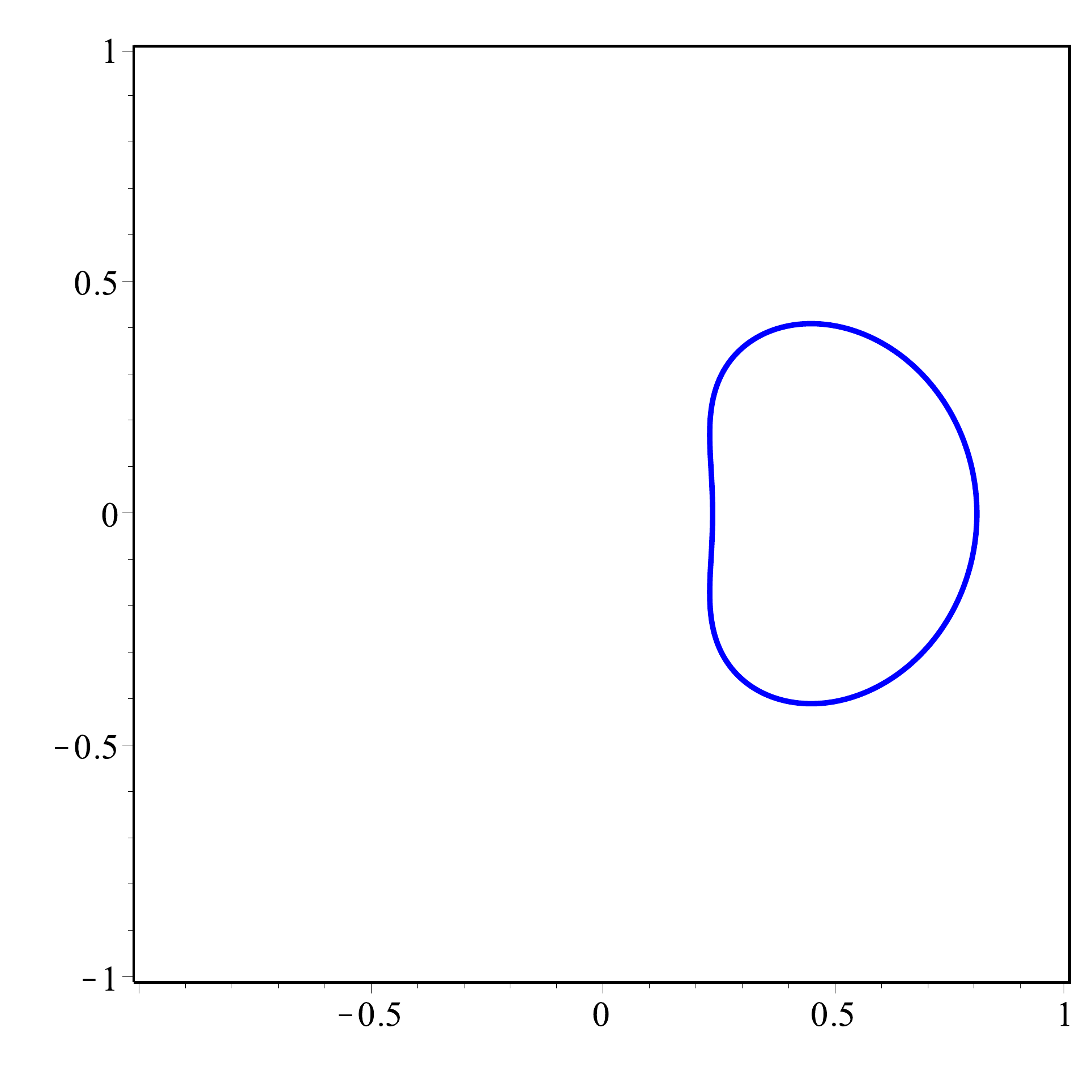}\tvalue{+0.2}\includegraphics[width=0.24\textwidth]{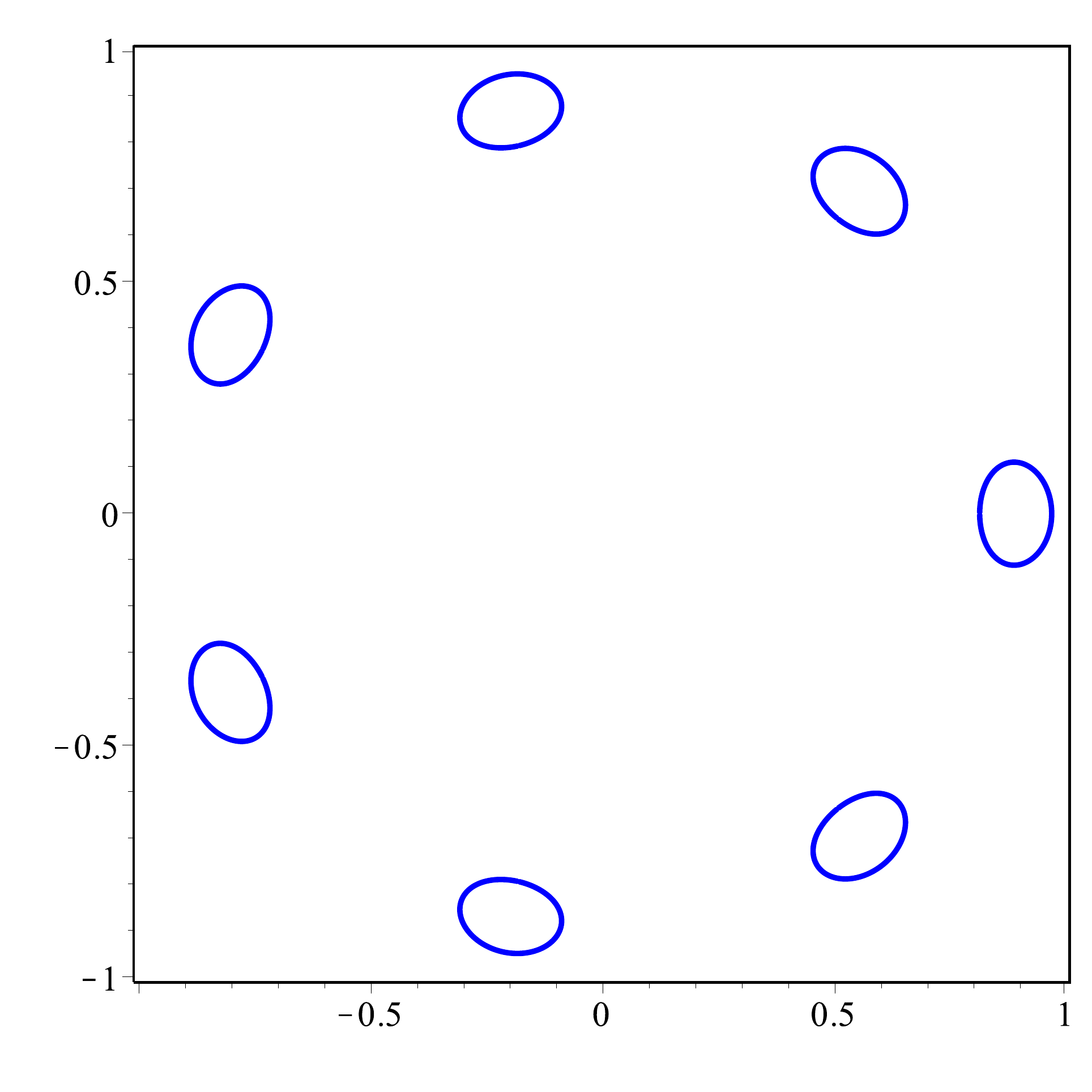}\\
{\bf Figure 3.} The equilibrium supports for $n=9$, $d=7$ ($Q(z)$ left, $V(z)$ right)
\end{center}
\newpage
\begin{remark}
It is easy to check that the parameters $r=r(t)$ and $\alpha=\alpha(t)$ are continuous in $t$. Hence the conformal map, as a function of $t$, is also continuous:
\be
\lim_{|t|\to t_{cr}+}f(u) =r_{cr}\left(u-\frac{1}{\overline{\alpha_{cr}}}\right)\left(1-\frac{\alpha_{cr}}{u}\right)^{\frac{d}{n}-1}= r_{cr}u\left(1-\frac{\alpha_{cr}}{u}\right)^{\frac{d}{n}}=\lim_{|t|\to t_{cr}-}f(u)\ .
\ee
For $|t|< t_{cr}$ the point $z=0$ belongs to the support of the equilibrium measure and at $|t|=t_{cr}$ the origin is on the boundary of the support. 
\end{remark}
A straightforward application of Lemma \ref{lemma:z_d_eq} to the potential $Q(z)$ in Theorem \ref{thm:main} yields the
\begin{shaded}
\begin{corollary} Let $0<d<n$ or $n<d<2n$.
The equilibrium measure for the potential
\be
V(z)=\frac{1}{T}\left(|z|^{2n}-tz^{d}-\bar{t}\bar{z}^{d}\right)
\ee
is given by $\mu_{Q}^{(d)}$, the $d$-fold rotated equilibrium measure of $Q$ given above.

More precisely, $\mu_V$ is absolutely continuous with respect to the area measure with density
\be
d\mu_V(z) = \frac{n^2}{\pi T}|z|^{2n-2}\chi_{C}(z)dA(z)\ ,
\ee
and the support set $C$ of $\mu_V$ can be described as follows:
\begin{itemize}
\item For $|t|\leq t_{cr}$, $C$ is simply connected and it is given by the exterior conformal map
\be
g(u) = \left(f(u^d)\right)^{\frac{1}{d}} = r^{\frac{1}{d}}u\left(1-\frac{\alpha}{u^d}\right)^{\frac{1}{n}}\ ,
\ee
where $r$ and $\alpha$ are given in Theorem \ref{thm:main}.
\item For $|t|>t_{cr}$, $C$ has $d$ disjoint simply connected components and their boundary is parametrized by
\be
g(u) = r^{\frac{1}{d}}u\left(1-\frac{1}{\bar{\alpha}u^d}\right)^{\frac{1}{d}}\left(1-\frac{\alpha}{u^d}\right)^{\frac{1}{n}-\frac{1}{d}}\ ,
\ee
where $u$ lies on the unit circle and  the factor
\be
\left(1-\frac{1}{\bar{\alpha}u^d}\right)^{\frac{1}{d}} \sim 1 +{\mathcal O}\left(\frac{1}{u}\right) \quad u \to \infty
\ee
is defined with some suitably defined branch cuts on $|u|\geq 1$, with $r$ and $\alpha$ given in Theorem \ref{thm:main}.
\end{itemize}
\end{corollary}
\end{shaded}

Note that the equilibrium problem for the class of potentials considered in the present paper was studied in \cite{Etingof_Ma} using conformal mapping techniques. Their method is based on a singularity correspondence result originating in the works of Richardson \cite{Richardson} and Gustafsson \cite{Gustafsson} and developed further by Entov and Etingof in \cite{Entov_Etingof}. Etingof and Ma obtained the functional form of the exterior conformal map for small values of $|t|$ only relying on a pertubative approach. 

As an original contribution, we find the critical value of the parameter $t$ where the pre-critical form of the conformal map \eqref{eq:conformal_map} breaks down and we obtain the post-critical form of $f(u)$ above. We prove explicitly that the variational inequalities \eqref{eq:var_eq} and \eqref{eq:var_ineq} hold for all admissible values of the parameters for the reduced potentials \eqref{ex_pot_reduced}. The analogous results for the class of potentials \eqref{ex_pot} are obtained by applying the symmetry reduction described above.

The paper is organized as follows: in Section \ref{sec:z_d} the necessary preliminaries on discrete rotational symmetries and the proof of Lemma \ref{lemma:z_d_eq} are given. In Section \ref{sec:entov_etingof} we give a short survey of the Entov-Etingof singularity correspondence that we need to construct the conformal maps. In the following sections we develop the proof of Theorem \ref{thm:main}: in Section \ref{sec:conf_map} we find the structure of the parametrizing conformal maps, in Sections \ref{sec:eq_pre} and \ref{sec:eq_post} the critical value $t_{cr}$ is calculated and the equations for the parameters of the pre- and post-critical conformal map are given. In Section \ref{sec:proof} we conclude the proof of the theorem showing that the variational inequalities \eqref{eq:var_eq} and \eqref{eq:var_ineq} hold.


\section{Discrete rotational symmetries}
\label{sec:z_d}

The \emph{logarithmic potential} and the \emph{Cauchy transform} of a measure $\mu$ are defined as
\be
\label{udef}
U^{\mu}(z) = \int\log\frac{1}{|z-w|}d\mu(w)\ ,
\ee
and
\be
C_{\mu}(z) = \int\frac{d\mu(w)}{w-z}\ ,
\ee
respectively. 

\begin{proposition}
The $d$-fold rotated measure $\mu^{(d)}$ of the measure $\mu$ has the logarithmic potential
\be
U^{\mu^{(d)}}(z) = \frac{1}{d}U^{\mu}(z^{d})
\ee
and the Cauchy transform
\be
C_{\mu^{(d)}}(z) = z^{d-1}C_{\mu}(z^d)\ .
\ee
\end{proposition}
\begin{proof}
\begin{align}
U^{\mu^{(d)}}(z) &= \frac{1}{d}\sum_{k=0}^{d-1}\int_{S_k}\log\frac{1}{|z-w|}d\mu^{(d)}_k(w)\\
&= \frac{1}{d}\sum_{k=0}^{d-1}\int_{\C}\log\frac{1}{|z-\varphi_k(u)|}d\mu(u)\ .
\end{align}
Since
\be
\prod_{k=0}^{d-1}(z-\varphi_k(u)) = z^d-u\ ,
\ee
We have that
\begin{equation}
U^{\mu^{(d)}}(z) = \frac{1}{d}\int_{\C}\log\frac{1}{|z^d-u|}d\mu(u) = \frac{1}{d}U^{\mu}(z^d)\ .
\ee
The equation involving the Cauchy transforms can be proved similarly.
\end{proof}

\begin{proof}[Proof of Lemma \ref{lemma:z_d_eq}]
Let $\mu_Q$ be the equilibrium measure of the admissible potential $Q(z)$, which is uniquely characterized by the following variational inequalities for some constant $F$:
\begin{align}
&Q(z) + 2 U^{\mu_Q}(z) = F \quad z \in \supp(\mu_Q) \quad \mbox{ q.e. }\\
&Q(z) + 2 U^{\mu_Q}(z) \geq F \quad z \not \in \supp(\mu_Q)\quad \mbox{ q.e. }
\end{align}
Therefore
\begin{align}
&Q(z^d) + 2 U^{\mu_Q}(z^d) = F \quad z^d \in \supp(\mu_Q)\quad \mbox{ q.e. }\\
&Q(z^d) + 2 U^{\mu_Q}(z^d) \geq F \quad z^d \not \in \supp(\mu_Q)\quad \mbox{ q.e. }
\end{align}
Using the previous proposition with $\mu=\mu_Q^{(d)}$, we find that
\begin{align}
&V(z) + 2 U^{\mu_{Q}^{(d)}}(z) = \frac{F}{d} \quad z \in \supp(\mu_Q^{(d)})\quad \mbox{ q.e. }\\
&V(z) + 2 U^{\mu_{Q}^{(d)}}(z) \geq \frac{F}{d} \quad z \not \in \supp(\mu_Q^{(d)})\quad \mbox{ q.e. }\ ,
\end{align}
where
\be
Q(z) = \frac{1}{d}V(z^{d})\ .
\ee
This is enough to prove that $\mu_V = \mu_Q^{(d)}$.
\end{proof}

\section{Singularity correspondence}
\label{sec:entov_etingof}

We briefly recall the method developed by Entov and Etingof in \cite{Entov_Etingof} based on the singularity correspondence of Richardson \cite{Richardson} and Gustafsson \cite{Gustafsson} to find the functional form of the exterior uniformizing map of the support of the equilibrium measure.

To this end, consider a polynomial perturbation of a radially symmetric potential of the form
\be
\label{vdef}
V(z) = W(z\bar{z})-P(z)-\overline{P(z)}\ ,
\ee
where $W(x)$ is a real-valued twice continuously differentiable function on $(0, \infty)$ and $P(z)$ is a polynomial of degree $m$ such that the potential $V(z)$ is admissible.
We seek the equilibrium measure in the form
\be
d\mu(z)=\frac{1}{\pi}\chi_K(z)\rho(z\bar{z})dA(z)\ ,
\ee
where 
\be
\rho(z\bar{z})=\frac{1}{4}\Delta W(z\bar{z}) = W''(z\bar{z})z\bar{z}+W'(z\bar{z})\ ,
\ee
and we assume that $K$ is simply connected and $\partial K$ is a Jordan curve with interior and exterior $D_{+}$ and $D_{-}$ respectively.
For any $z\in D_{+}$ we have
\be
\label{eq:stokes}
W'(z\bar{z})\bar{z}=\frac{1}{2\pi i}\int_{\partial K}\frac{W'(w\bar{w})\bar{w}dw}{w-z}+\frac{1}{2\pi i} \int_{K}\frac{\rho(w\bar{w})dw \wedge d\bar{w}}{w-z}
\ee
by Stokes' Theorem. Consider the Sokhotsky-Plemelj decomposition
\be
\label{eq:sok_ple}
W'(z\bar{z})\bar{z} = \varphi_{+}(z)-\varphi_{-}(z) \qquad z \in \partial K\ ,
\ee
with
\be
\varphi_{\pm}(z) = \frac{1}{2\pi i}\int_{\partial K}\frac{W'(w\bar{w})\bar{w}dw}{w-z} \quad z \in D_{\pm}
\ee
holomorphic in $D_{\pm}$ respectively and $\varphi_{-}(z) \to 0$  for $z \to \infty$.
From \eqref{eq:stokes},
\be
\label{split1}
\frac{1}{\pi}\int_{K}\frac{\rho(w\bar{w})dA(w)}{w-z} = 
\left\{\begin{array}{ll}
\varphi_{+}(z)-W'(z\bar{z})\bar{z} & z \in D_{+}\\
\varphi_{-}(z) & z \in D_{-}\ .
\end{array}
\right.
\ee
By the variational equality \eqref{eq:var_ineq} the effective potential
\be
E(z) = V(z)+2U_{\mu}(z)
\ee
is constant on $K$ and therefore its $\partial_{z}$-derivative gives, by \eqref{udef} and \eqref{split1},
\be
\label{pprime}
P'(z) = \varphi_{+}(z)\ .
\ee
This equation determines the shape of $K$ through its external conformal map, as shown below.

In what follows, we assume that the uniformizing map $z=f(u)$ extends continuously to the boundary $|u|=1$.
The Sokhotsky-Plemelj decomposition \eqref{eq:sok_ple} implies the identity
\be
\varphi_{+}(f(u))-\varphi_{-}(f(u)) = W'\left(f(u)\overline{f(u)}\right)\overline{f(u)} \qquad |u|=1
\ee
on the unit circle of the $u$-plane. Given that
\be
\overline{f(u)} = \bar{f}\left(\frac{1}{u}\right) \quad |u|=1\ ,
\ee
multiplying both sides by $f(u)$ and rearranging gives that
\be
\label{eq:entov_etingof}
f(u)\varphi_{+}(f(u))-H(u) = f(u)\varphi_{-}(f(u)) \qquad |u|=1\ ,
\ee
where
\be
\label{eq:hu}
H(u)= W'\left(f(u)\bar{f}\left(\frac{1}{u}\right)\right)f(u)\bar{f}\left(\frac{1}{u}\right)\ .
\ee
The main idea of the singularity correspondence is that since $f(u)\varphi_{-}(f(u))$ is holomorphic outside the unit circle $|u|=1$, the l.h.s of \eqref{eq:entov_etingof}, seen as a function on the unit circle, possesses an analytic continuation outside the unit disk of the $u$-plane.
Therefore the analytic continuations of the functions $f(u)\varphi_{+}(f(u))$ and $H(u)$ have the same singularities in the exterior of the unit disk on the $u$-plane.

Since $f(u)$ is a univalent uniformizing map, the singularities of $z\varphi_{+}(z)$ in $D_{-}$ are in one-to-one correspondence with the singularities of $f(u)\varphi_{+}(f(u))$ in $|u|>1$ through the mapping $f$. This establishes a $1-1$ correspondence between the singularities of $z\varphi_{+}(z)$ in $D_{-}$ and $H(u)$ in $|u|>1$.  In particular, $H(u)$ has a pole of order $k$ at $u = u_0$ in $|u|>1$ if and only if $z\varphi_{+}(z)$ has a pole of order $k$ at $z_0 = f(u_0)$.

Since $z\varphi_{+}(z) = zP'(z)$ is a polynomial of degree $m$, it only has a pole of order $m$ at $z=\infty$, and hence the only singularity of $H(u)$ in $|u|>1$ is a pole of order $m$ at $u =\infty$.
Since
\be
\overline{H(u)} = \bar{H}\left(\frac{1}{u}\right)=H(u) \qquad |u|=1\ ,
\ee
$H(u)$ is real on the unit circle and it can be analytically continued using the Schwarz reflection principle:
\be
\tilde H(u)
=\left\{\begin{array}{ll}
\displaystyle H(u) & |u|>1\\
\displaystyle \bar{H}\left(\frac{1}{u}\right) & |u|<1\ .
\end{array}
\right.
\ee
Therefore $\tilde H(u)$ is a Laurent polynomial of degree $m$ which is real-valued and positive on the unit circle. By the Fej\'er-Riesz Lemma, there exists a polynomial $R(u)$ of degree $m$ with zeroes inside the unit disk such that
\be
\label{hr}
H(u) = R(u)\bar{R}\left(\frac{1}{u}\right)\ .
\ee

\section{The conformal map}
\label{sec:conf_map}
Now we specialize to the class of potentials \eqref{ex_pot_reduced}, where
\be
W(x) = \frac{d}{T}x^{\frac{n}{d}}\quad \mbox{and}\quad P(z) = \frac{td}{T}z\ .
\ee
It is important to distinguish between two different regimes corresponding to small values of $|t|$ and large values of $|t|$. If $|t|$ is small enough we expect that $z=0$ belongs to the support $K$ since the equilibrium measure for $t=0$ is supported on a disk centered at the origin. For large $|t|$, however, the constant field coming from the perturbative term $tz+\bar{t}\bar{z}$ becomes strong enough so that the equilibrium domain separates from the origin. We expect a critical transition when $z \in \partial K$  that will affect the functional form of the conformal map. Hence the cases $0 \in K$ and $0 \in \C\setminus K$ are referred to as \emph{pre-critical} and \emph{post-critical}, respectively.

\subsubsection*{Pre-critical case}
Following \cite{Etingof_Ma}, we seek the conformal map in the form
\be
f(u) = ru g(u)\ ,
\ee
where $g$ is holomorphic in $|u|>1$ with
\be
g(u) = 1+{\mathcal O}\left(\frac{1}{u}\right) \quad u \to \infty\ ,
\ee
and
\be
g(u)\not=0 \qquad |u|>1
\ee
since $z=0$ is assumed to belong to $K$. The singularity correspondence \eqref{hr} implies that
\be
\label{eq:e_e}
\frac{n}{T}\left[f(u)\bar{f}\left(\frac{1}{u}\right)\right]^{\frac{n}{d}} = R(u)\bar{R}\left(\frac{1}{u}\right)
\ee
for some linear polynomial $R(u)$ and therefore
\be
\frac{n r^{\frac{2n}{d}}}{T}\left[g(u)\bar{g}\left(\frac{1}{u}\right)\right]^{\frac{n}{d}}=R(u)\bar{R}\left(\frac{1}{u}\right)\ .
\ee
The polynomial $R$ is parametrized as
\be
R(u) = r^{\frac{n}{d}}\sqrt{\frac{n}{T}} (u-\alpha)\ ,
\ee
where $|\alpha| \leq 1$ is an unknown parameter, to be determined from the data. Therefore $g$ is of the form
\be
g(u) = \left(1-\frac{\alpha}{u}\right)^{\frac{d}{n}}\ ,
\ee
and hence
\begin{shaded}
\be
\label{eq:c_map_pre}
f(u) = ru\left(1-\frac{\alpha}{u}\right)^{\frac{d}{n}}\ .
\ee
\end{shaded}
To find the equlibrium support we have to determine the correct values of the parameters $r$ and $\alpha$.

\subsubsection*{Post-critical case} If the $z=0$ is outside $K$ then $f(u)$ must have a zero in $|u|>1$. To make sense of \eqref{eq:e_e}, the ansatz presented in \cite{Etingof_Ma} needs to be modified
to account for the vanishing of the conformal map at, say, $u=\beta$. This can be done by introducing a \emph{Blaschke factor}, i.e., 
\be
u\mapsto \frac{u-\beta}{1-\bar{\beta}u}\qquad |\beta|>1\ ,
\ee
that leaves the unit circle of the $u$-plane invariant and swaps the interior and the exterior. This function can be factored into the proposed form of the conformal map:
\be
f(u) = ru \left(-\bar{\beta}\right)\frac{u-\beta}{1-\bar{\beta}u}g(u)
\ee
where $g$ is holomorphic in $|u|>1$ with
\be
g(u)\not=0 \qquad |u|>1\ ,\qquad g(u) = 1+{\mathcal O}\left(\frac{1}{u}\right) \quad u \to \infty\ .
\ee
Then
\be
f(u)\bar{f}\left(\frac{1}{u}\right)=r^{2}|\beta|^{2}\frac{u-\beta}{1-\bar{\beta}u}g(u)\frac{1-\bar{\beta}u}{u-\beta}\bar{g}\left(\frac{1}{u}\right)=r^{2}|\beta|^{2}g(u)\bar{g}\left(\frac{1}{u}\right)
\ee
and therefore
\be
\frac{n r^{\frac{2n}{d}}|\beta|^{\frac{2n}{d}}}{T}\left[g(u)\bar{g}\left(\frac{1}{u}\right)\right]^{\frac{n}{d}}=R(u)\bar{R}\left(\frac{1}{u}\right)\ .
\ee
Now
\be
R(u) = r\beta\left(\frac{n}{T}\right)^{\frac{d}{2n}}(u-\alpha)\ ,
\ee
where $|\alpha| <1$ is unknown. The function $g$ is again
\be
g(u) = \left(1-\frac{\alpha}{u}\right)^{\frac{d}{n}}\ ,
\ee
and hence
\begin{shaded}
\be
\label{eq:c_map_post}
f(u) = ru\frac{u-\beta}{u-\frac{1}{\bar{\beta}}}\left(1-\frac{\alpha}{u}\right)^{\frac{d}{n}}\ ,
\ee
\end{shaded}
\noindent
where the parameters $r, \alpha$ and $\beta$ are to be determined from the data.
\begin{remark}
It is to be checked that these functions are univalent maps outside the unit circle: for this we refer to Appendix \ref{sec:univalency}.
\end{remark}

\section{Equations for the parameters in the pre-critical case}
\label{sec:eq_pre}
In this section we take a conformal map of the form \eqref{eq:c_map_pre} as an ansatz and derive equations for the parameters $r$ and $\alpha$.
The analysis of these equations allows to find the critical value $|t|=t_{cr}$ for which the condition $|\alpha|<1$ is no longer satisfied; it will be shown that the value $|t|=t_{cr}$ separates the pre- and post-critical cases.

\noindent
{\bf Total mass condition.} 

\begin{align}
\frac{1}{\pi}\int_{K}\rho(w\bar{w})dA(w) &= \frac{n}{T}\frac{1}{2\pi i}\int_{\partial K}w^{\frac{n}{d}}\bar{w}^{\frac{n}{d}}\frac{dw}{w}\\
&= \frac{n}{T}\frac{1}{2\pi i}\int_{|u|=1}\left[f(u)\bar{f}\left(\frac{1}{u}\right)\right]^{\frac{n}{d}}\frac{f'(u)}{f(u)}du\\
&= \frac{r^{\frac{2n}{d}}n}{T}\frac{1}{2\pi i}\int_{|u|=1}\left(1-\frac{\alpha}{u}\right)\left(1-\bar{\alpha}u\right)\left[\left(1-\frac{d}{n}\right)\frac{1}{u} +\frac{d}{n}\frac{1}{u-\alpha}\right]du\\
&=\frac{r^{\frac{2n}{d}}}{T}\left(n+(n-d)|\alpha|^{2}\right)\ .
\end{align}
In the calculation above we assumed that $|\alpha|<1$.
In the chosen normalization of the energy problem the equilibrium measure $\mu_Q$ is a probability measure, and this gives the following equation in $r$ and the modulus of $\alpha$:
\be
\label{eq:area_pre}
\frac{r^{\frac{2n}{d}}}{T}\left(n+(n-d)|\alpha|^{2}\right)=1\ .
\ee

\noindent
{\bf Deformation condition.} From \eqref{eq:hu} we can compute
\be
\frac{H(u)}{f(u)}=\frac{nr^{\frac{2n}{d}}}{T}\left(1-\frac{\alpha}{u}\right)\left(1-\bar{\alpha}u\right)\frac{1}{f(u)}
=-\frac{nr^{\frac{2n}{d}-1}\bar{\alpha}}{T} +{\mathcal O}\left(\frac{1}{u}\right) \quad u \to \infty\ .
\ee
Since we know $\varphi_{+}$ from \eqref{pprime}, by using \eqref{eq:entov_etingof} and the previous asympotic behaviour we can obtain
\begin{align}
\label{phi+}
\varphi_{+}(f(u))&=-\frac{nr^{\frac{2n}{d}-1}\bar{\alpha}}{T}\\
\varphi_{-}(f(u))&=-\frac{nr^{\frac{2n}{d}-1}\bar{\alpha}}{T}-\frac{nr^{\frac{2n}{d}}}{T}\left(1-\frac{\alpha}{u}\right)\left(1-\bar{\alpha}u\right)\frac{1}{f(u)}\ .
\end{align}
Comparison of \eqref{phi+} with \eqref{pprime} gives the second equation
\be
\label{eq:def_pre}
t = -\frac{n}{d}r^{\frac{2n}{d}-1}\bar{\alpha}\ .
\ee
Combining \eqref{eq:area_pre} and \eqref{eq:def_pre} gives the equation
\be
\label{eq:r_eq}
r^{\frac{4n}{d}-2}-\frac{T}{n}r^{\frac{2n}{d}-2}+\frac{n-d}{n}\frac{d^2}{n^2}|t|^2=0
\ee
for the conformal radius $r$ as a function of $t$. It can be shown (see Appendix \ref{sec:conf_r}) that \eqref{eq:r_eq} has a unique positive solution $r=r_0$ such that
\be
|\alpha(r_0)|<1\ .
\ee
Equations \eqref{eq:r_eq} and \eqref{eq:def_pre} are what we need to characterize all the parameters in the pre-critical case.

\section{Equations for the parameters for the post-critical case}
\label{sec:eq_post}
For $|t|>t_{cr}$ we take the modified form of the conformal map \eqref{eq:c_map_post} that allows the description of the support that does not contain the point $z=0$. 
As shown below, there are two equations for the three unknowns $r, \alpha$ and $\beta$ which are insufficient to characterize the conformal map. However, the uniqueness of the equilibrium measure implies that there is only one choice of the parameters which corresponds to this unique measure. In the next section to resolve this selection problem we make an ansatz on the value of $\beta$ for which the corresponding measure satisfies not only the total mass and the deformation equations but also the variational inequalities.

\noindent
{\bf Total mass condition.} Now the total mass of the measure is given by
\begin{align}
\frac{1}{\pi}\int_{K}\rho(w\bar{w})dA(w) &= \frac{|\beta|^{\frac{2n}{d}}r^{\frac{2n}{d}}n}{T}\frac{1}{2\pi i}\int_{|u|=1}\left(1-\frac{\alpha}{u}\right)\left(1-\bar{\alpha}u\right)\frac{f'(u)}{f(u)} du\\
&=\frac{n}{T}|\beta|^{\frac{2n}{d}}r^{\frac{2n}{d}}\left(\frac{\bar{\alpha}}{\bar{\beta}}+\frac{\alpha}{\beta}-\frac{d}{n}|\alpha|^{2}\right)\ .
\end{align}
This gives the following equation on $r$ and the modulus of $\alpha$:
\be
\label{eq:area_post}
\frac{n}{T}|\beta|^{\frac{2n}{d}}r^{\frac{2n}{d}}\left(\frac{\bar{\alpha}}{\bar{\beta}}+\frac{\alpha}{\beta}-\frac{d}{n}|\alpha|^{2}\right)=1\ .
\ee
\noindent
{\bf Deformation condition.} Now
\be
\frac{H(u)}{f(u)}=  \frac{n}{T}|\beta|^{\frac{2n}{d}}r^{\frac{2n}{d}}\left(1-\frac{\alpha}{u}\right)\left(1-\bar{\alpha}u\right)\frac{1}{f(u)}=-\frac{n}{T}|\beta|^{\frac{2n}{d}}r^{\frac{2n}{d}-1}\bar{\alpha}+{\mathcal O}\left(\frac{1}{u}\right) \quad u \to \infty\ ,
\ee
and therefore, applying the same comparison as in the pre-critical case we get
\begin{align}
\varphi_{+}(f(u))&=-\frac{n}{T}|\beta|^{\frac{2n}{d}}r^{\frac{2n}{d}-1}\bar{\alpha}\\
\varphi_{-}(f(u))&=-\frac{n}{T}|\beta|^{\frac{2n}{d}}r^{\frac{2n}{d}-1}\bar{\alpha}
-\frac{n}{T}|\beta|^{\frac{2n}{d}}r^{\frac{2n}{d}}\left(1-\frac{\alpha}{u}\right)\left(1-\bar{\alpha}u\right)\frac{1}{f(u)}\ .
\end{align}
This implies, by \eqref{pprime}, that
\be
\label{eq:def_post}
t = -\frac{n}{d}|\beta|^{\frac{2n}{d}}r^{\frac{2n}{d}-1}\bar{\alpha}\ .
\ee
Equations \eqref{eq:area_post} and \eqref{eq:def_post} are not enough to determine all the parameters. To bypass this difficulty we make an ansatz on the parameter $\beta$ and later on we will justify that the choice made is correct.

\subsection{Ansatz on $\beta$ and the parameters of the conformal map}
\label{sec:ansatz}
To get an idea what the value of $\beta$ should be, it is natural, at first, to look at the case $d=n$. As we have seen above, the corresponding conformal map
is linear for all values of $t$:
\be
f(u)=r\left(u+\frac{\bar{t}}{r}\right)\ ,\qquad r=\sqrt{\frac{T}{n}}\ .
\ee
The post-critical form \eqref{eq:c_map_post} of the conformal map gives
\be
f(u)=r\frac{u-\beta}{u-\frac{1}{\bar{\beta}}} (u-\alpha)\ ,
\ee
which is linear if and only if $\beta=\frac{1}{\bar{\alpha}}$. This suggests to make the following ansatz:
\begin{shaded}
\be
\beta=\frac{1}{\bar{\alpha}}\ .
\ee
\end{shaded}
With this assumption, the total mass and deformation conditions assume the following form:
\begin{align}
\label{area_ab}
\frac{n}{T}r^{\frac{2n}{d}}\frac{2n-d}{n}|\alpha|^{2-\frac{2n}{d}}&=1\\
\label{def_ab}
-\frac{n}{d}r^{\frac{2n}{d}-1}|\alpha|^{-\frac{2n}{d}}\bar{\alpha}&=t\ .
\end{align}
\begin{proposition}
\label{prop:r_post} 
Assuming $|t|>t_{cr}$, the following explicit formulae hold:
\be
r=\left(\frac{T}{2n-d}\right)^{\frac{1}{2}}\left(\frac{d}{n}|t|\right)^{\frac{d-n}{2n-d}}\ ,
\ee
and
\be
\alpha= -\frac{\bar{t}}{t}\left(\frac{T}{2n-d}\right)^{\frac{1}{2}}\left(\frac{d}{n}|t|\right)^{-\frac{n}{2n-d}}\ .
\ee
Moreover, the condition $|\alpha|<1$ is satisfied for all $|t|>t_{cr}$.
\end{proposition}
\begin{proof}
From Eq.~\eqref{def_ab} we obtain
\be
|\alpha|=\left(\frac{d}{n}|t|\right)^{-\frac{d}{2n-d}}r\ ,
\ee
which, with equation \eqref{area_ab} gives the explicit form for $r$ and $\alpha$ above. 
The condition
\be
|\alpha|=\left(\frac{T}{2n-d}\right)^{\frac{1}{2}}\left(\frac{d}{n}|t|\right)^{-\frac{n}{2n-d}}<1
\ee
rewritten in terms of $|t|$ is equivalent to
\be
|t|>\frac{n}{d}\left(\frac{T}{2n-d}\right)^{\frac{2n-d}{2n}}=t_{cr}\ ,
\ee
i.e., $|\alpha|<1$ is always satisfied in the post-critical regime $|t|>t_{cr}$.
\end{proof}

\section{Variational inequalities}
\label{sec:proof}
Now we are ready to conclude the proof the main result of this paper.

According to Prop.~\ref{prop:r_pre}, for $0<|t|<t_{cr}$ a conformal map of the form \eqref{eq:c_map_pre} satisfies the assertions of the theorem.
The complementary result Prop.~\ref{prop:r_post} shows that for $|t|>t_{cr}$ the ansatz \eqref{eq:c_map_post} with $\beta=\frac{1}{\bar{\alpha}}$ leads to a system of equations for the parameters that have a solution with the stated properties.

To conlude the proof we need to show that the variational inequality \eqref{eq:var_ineq} holds in both cases for the proposed measures.

We have
\be
\partial_{z}E(z) = W'(z\bar{z})\bar{z}-P'(z) + \varphi_{-}(z)\qquad z \in D_{-}\ .
\ee
This is identically zero on $\partial K$.
Since $E(z)=F$ on the boundary of $K$ and 
\be
E(z) \to \infty \qquad |z|\to \infty\ ,
\ee
the effective potential can go below the value $F$ in $D_{-}$ only if $E(z)$ has a critical point $z_0 \in D_{-}$ such that
\be
\label{critical_point_eq}
\left.\partial_{z}E(z)\right|_{z=z_0} = W'(z_0\bar{z_0})\bar{z_0}-P'(z_0) + \varphi_{-}(z_0)=0\ .
\ee
Therefore to prove the variational inequality \eqref{eq:var_ineq} it is sufficient to show that the derivative does not have any critical points in $D_{-}$.

To this end, it is enough to show that \eqref{critical_point_eq} is never satisfied, i.e., in terms of the uniformizing coordinate $u$,
\be
\partial_{z}E(z)|_{z=f(u)}=W'(f(u)\overline{f(u)})\overline{f(u)}-P'(f(u))+\varphi_{-}(f(u))\not=0 \qquad |u|>1\ .
\ee
{\bf Pre-critical case.}
\begin{align}
\partial_{z}E(z)|_{z=f(u)}&=W'(f(u)\overline{f(u)})\overline{f(u)}-\frac{nr^{\frac{2n}{d}}}{T}\left(1-\frac{\alpha}{u}\right)\left(1-\bar{\alpha}u\right)\frac{1}{f(u)}\\
&=\frac{nr^{\frac{2n}{d}}}{T}\left(1-\frac{\alpha}{u}\right)\frac{1}{f(u)}\left[|u|^{\frac{2n}{d}}\left(1-\frac{\bar{\alpha}}{\bar{u}}\right)-(1-\bar{\alpha}u)\right]\ .
\end{align}
Since $|\alpha|<1$, this expression vanishes for some $u$ with $|u| > 1$ only if
\be
|u|^{\frac{2n}{d}}\left(1-\frac{\bar{\alpha}}{\bar{u}}\right)-(1-\bar{\alpha}u)=0\ ,
\ee
which implies that
\be
|u|^{\frac{2n}{d}-1}=\left|\frac{1-\bar{\alpha}u}{u-\alpha}\right|\ .
\ee
This is impossible since
\be
|u|^{\frac{2n}{d}-1} >1 \quad \mbox{ and }\quad \left|\frac{1-\bar{\alpha}u}{u-\alpha}\right| < 1\ .
\ee
{\bf Post-critical case.} Here the derivative of the effective potential is
\begin{align}
\partial_{z}E(z)|_{z=f(u)}&= W'(f(u)\overline{f(u)})\overline{f(u)}-\frac{nr^{\frac{2n}{d}}}{T|\alpha|^{\frac{2n}{d}}}\left(1-\frac{\alpha}{u}\right)\left(1-\bar{\alpha}u\right)\frac{1}{f(u)}\\
&=\frac{nr^{\frac{2n}{d}}}{T|\alpha|^{\frac{2n}{d}}}\left(1-\frac{\alpha}{u}\right)\frac{1}{f(u)}\left[|u|^{\frac{2n}{d}}\left|\frac{1-\bar{\alpha}u}{u-\alpha}\right|^{\frac{2n}{d}}\left(1-\frac{\bar{\alpha}}{\bar{u}}\right)-(1-\bar{\alpha}u)\right]\ .
\end{align}
Since $|\alpha|<1$, this may vanish for some $u$ in the exterior of the unit disk only if
\be
|u|^{\frac{2n}{d}}\left|\frac{1-\bar{\alpha}u}{u-\alpha}\right|^{\frac{2n}{d}}\left(1-\frac{\bar{\alpha}}{\bar{u}}\right)-(1-\bar{\alpha}u)=0\ ,
\ee
which implies
\be
\left|u\frac{1-\bar{\alpha}u}{u-\alpha}\right|=1\ .
\ee
As a consequence we have
\be
1-\bar{\alpha}u=1-\frac{\bar{\alpha}}{\bar{u}}\ ,
\ee
and thus $|u|=1$, a contradiction. The variational inequality is hence satisfied in both cases and this concludes the proof of the assertions of Theorem \ref{thm:main}.

\section*{Acknowledgements}
D. M. is grateful to K. McLaughlin for suggesting the problem and for the support provided during the initial stages of the present work.
The authors would like to thank T. Grava for helpful comments and suggestions.

The present work was supported by the FP7 IRSES project  RIMMP \emph{Random and Integrable models in Mathematical Physics 2010-2014},
the ERC project \emph{FroM-PDE Frobenius Manifolds and Hamiltonian Partial Differential Equations  2009-13} and the 
MIUR Research Project  \emph{Geometric and analytic theory of Hamiltonian systems in finite and infinite dimensions}.

\appendix
\section{Univalency}
\label{sec:univalency}
\subsubsection*{Pre-critical case}
\begin{definition}
A map $f:\{|u|>1\}\rightarrow \mathbb{C}$ is \emph{starlike} if is univalent and the compact complement of its image is star-shaped w.r.t. the origin.
\end{definition}
\begin{theorem}[in \cite{Pommenrenke}]
Let $f:\{|u|>1\}\rightarrow \mathbb{C}$, then  $f$ is starlike iff
\begin{equation}
\label{univ_cond}
\Re\left(u\frac{f'(u)}{f(u)}\right)>0 \qquad |u|>1\ .
\end{equation}
\end{theorem}
In order to prove that our pre-critical map 
\begin{equation}
f(u)=ru\left(1-\frac{\alpha}{u}\right)^{\frac{d}{n}}
\end{equation}
is univalent, it suffices to show that \eqref{univ_cond} holds, i.e.
\begin{equation}
\Re\left(1+\frac{d}{n}\frac{\alpha}{u-\alpha}\right)>0 \ .
\end{equation}
Thus we just need to check that
\begin{equation}
\Re\left(\frac{\alpha}{u-\alpha}\right)>-\frac{1}{2} \quad \mbox{for}\ |u|>1\ .
\end{equation}
It is easy to see that $\frac{\alpha}{u-\alpha}$ maps the exterior of the circle of radius $|\alpha|$ into the halfplane $\Re(z)>-\frac{1}{2}$. This proves that our conformal map is univalent in $|u|>1$.

\subsubsection*{Post-critical case}
Let us consider the map
\begin{equation}
f(u)=-\frac{r}{\bar{\alpha}}\frac{1-\bar{\alpha}u}{u-\alpha}u\left(1-\frac{\alpha}{u}\right)^{\frac{d}{n}}\ .
\end{equation}
We can rescale $u\rightarrow \frac{|\alpha|}{\bar{\alpha}}u$ so that
\begin{equation}
f(u)=-\frac{r}{\bar{\alpha}}\frac{1-|\alpha|u}{u-|\alpha|}u\left(1-\frac{|\alpha|}{u}\right)^{\frac{d}{n}}=-\frac{r}{\bar{\alpha}}(1-|\alpha|u)\left(1-\frac{|\alpha|}{u}\right)^{\frac{d}{n}-1}\ .
\end{equation}

We need the following lemma:
\begin{lemma}[in \cite{Pommenrenke}]
\label{lemma:univ}
Let $f$ be a map analytic in $\{|u|\geq1\}$ and injective on $\{|u|=1\}$, then $f$ is univalent in $\{|u|\geq1\}$.
\end{lemma}
So we just need to prove that the image of unit circle has no self--intersections, i.e. $f(u)=f(v)$ iff $u=v$. We first look for the points $u=\mbox{e}^{i \mu}$ and $v=\mbox{e}^{i \nu}$  such that $|f(u)|=|f(v)|$:
\begin{align}
 \left|\frac{1-|\alpha|u}{u-|\alpha|}\right||u|\left|1-\frac{|\alpha|}{u}\right|^{\frac{d}{n}}&= \left|\frac{1-|\alpha|v}{v-|\alpha|}\right||v|\left|1-\frac{|\alpha|}{v}\right|^{\frac{d}{n}}\\
 \left|1-\frac{|\alpha|}{u}\right|^{\frac{d}{n}}&=\left|1-\frac{|\alpha|}{v}\right|^{\frac{d}{n}}\\
 \left|1-|\alpha|\mbox{e}^{-i \mu}\right|^{\frac{d}{n}}&=\left|1-|\alpha|\mbox{e}^{-i \nu}\right|^{\frac{d}{n}}\\
 \left|1-|\alpha|\mbox{e}^{-i \mu}\right|^2&=\left|1-|\alpha|\mbox{e}^{-i \nu}\right|^2\\
 2-2|\alpha|\cos(\mu)&=2-2|\alpha|\cos(\nu)\\
\cos(\mu)&=\cos(\nu)\ .
\end{align}
Thus $|f(\mbox{e}^{i \mu})|=|f(\mbox{e}^{i \nu})|$ iff $\nu=\pm\mu$.\\
Now we just need to show that $f(\mbox{e}^{i \mu})\neq f(\mbox{e}^{-i \mu})$ for $\mu\neq k\pi$.

\begin{align}
-\frac{r}{\bar{\alpha}}(1-|\alpha|\mbox{e}^{i \mu})\left(1-|\alpha|\mbox{e}^{-i \mu}\right)^{\frac{d}{n}-1}&=-\frac{r}{\bar{\alpha}}(1-|\alpha|\mbox{e}^{-i \mu})\left(1-|\alpha|\mbox{e}^{i \mu}\right)^{\frac{d}{n}-1}\\
(1-|\alpha|\mbox{e}^{i \mu})\left(1-|\alpha|\mbox{e}^{-i \mu}\right)^{\frac{d}{n}-1}&=(1-|\alpha|\mbox{e}^{-i \mu})\left(1-|\alpha|\mbox{e}^{i \mu}\right)^{\frac{d}{n}-1}\ .
\end{align}
The l.h.s. and the r.h.s. are complex conjugates, so this equation has solution iff the image of $\mbox{e}^{i \mu}$ through the map
\begin{equation}
\label{rea}
(1-|\alpha|\mbox{e}^{i \mu})\left(1-|\alpha|\mbox{e}^{-i \mu}\right)^{\frac{d}{n}-1}
\end{equation}
is real.\\
Let us change the variable $y:=1-|\alpha|\mbox{e}^{-i \mu}$: since $|\alpha|<1$, $y$ is contained in the unit disk centered at $1$, which is contained in the half plane $\Re(y)>0$. In terms of $y$, \eqref{rea} assumes the simple form
\begin{equation}
\label{rev}
\frac{\bar{y}}{y}y^{\frac{d}{n}}\ .
\end{equation}
The map $w=y^{\frac{d}{n}}$ sends the halfplane $\Re(y)>0$ into the set $\left\{w\in\mathbb{C}|-\frac{d}{2n}\pi<\arg(w)<\frac{d}{2n}\pi\right\}$, which, since $0<d<2n$, is contained in $\mathbb{C}\setminus\mathbb{R}_{-}$.\\ Written in its polar form, $y=\rho \mbox{e}^{i\lambda}$, \eqref{rev} becomes
\begin{equation}
\rho^{\frac{d}{n}}\mbox{e}^{-i\frac{2n-d}{n}\lambda}\ .
\end{equation}
Thus, since $-\frac{\pi}{2}<\lambda<\frac{\pi}{2}$, the previous number is real iff $\lambda=0$, i.e. $y\in\mathbb{R}$. This implies that $1-|\alpha|\mbox{e}^{-i\mu}\in\mathbb{R}$, but this can happen iff $\mu=k\pi$.

Therefore Lemma \ref{lemma:univ} gives that the post-critical map $f(u)$ is univalent in $|u|>1$.

\section{Analysis of the equation for the conformal radius}
\label{sec:conf_r}
Let us consider equation \eqref{eq:r_eq}
$$
r^{\frac{4n}{d}-2}-\frac{T}{n}r^{\frac{2n}{d}-2}+\frac{n-d}{n}\frac{d^2}{n^2}|t|^2=0\ .
$$
for the conformal radius $r$ as a function of $t$. We need to show that \eqref{eq:r_eq} has a unique positive solution $r=r_0$ such that
\be
\label{ineq_a}
|\alpha|=\frac{d}{n}|t|r_0^{1-\frac{2n}{d}} <1\ ,
\ee
or equivalently
\begin{equation}
\label{ineq_a1}
|\alpha|^{2}=\frac{n}{n-d}\left(\frac{T}{n}r_{0}^{-\frac{2n}{d}}-1\right) <1\ .
\end{equation}
Solving the critical equation $|\alpha|=1$ and using \eqref{eq:r_eq} we can obtain the critical values $t_{cr}$ and $r_{cr}:=r(t_{cr})$ given by
\begin{align}
t_{cr}&=\frac{n}{d}\left(\frac{T}{2n-d}\right)^{1-\frac{d}{2n}}\\
r_{cr}&=\left(\frac{T}{2n-d}\right)^{\frac{d}{2n}}.
\end{align}

Clearly the formulae above make sense only for $d\neq n$ and $d\neq 2n$. However, these cases are trivial so we can restrict ourselves to the study of the cases $0< d< n$ and $n<d<2n$.

\begin{proposition}
\label{prop:r_pre}
Assume that $|t|<t_{cr}$.
\begin{itemize}
\item
For $0< d<n$, Eq.~\eqref{eq:r_eq} has two positive solutions $r_{\pm}(|t|)$, with 
\be
0\leq r_{-}(|t|)<r_{+}(|t|)\leq  \left(\frac{T}{n}\right)^{\frac{d}{2n}} \quad \mbox{and}\quad  r_{-}(0)=0\ ,\ r_{+}(0)=  \left(\frac{T}{n}\right)^{\frac{d}{2n}}\ .
\ee 
With the choice $r=r_{+}(|t|)$ the inequality \eqref{ineq_a} is satisfied whereas the other solution $r=r_{-}(|t|)$ is not compatible with \eqref{ineq_a}.

\item For $n< d < 2n$ Eq.~\eqref{eq:r_eq} has a unique positive solution $r_{0}(|t|)$ that is compatible with the inequality \eqref{ineq_a}.
\end{itemize}
\end{proposition}
\begin{proof}
Let $0< d < n$. Consider the function
\be
y(r)= r^{\frac{4n}{d}-2}-\frac{T}{n}r^{\frac{2n}{d}-2}
\ee
on the non-negative real axis. The only roots of $y(r)=0$ are
\be
r=0 \quad \mbox{ and }\quad r= \left(\frac{T}{n}\right)^{\frac{d}{2n}}\ ,
\ee
and $y(r)$ has a unique minimum at
\be
r_{min}=\left(\frac{T}{n}\frac{n-d}{2n-d}\right)^{\frac{d}{2n}}
\ee
with
\be
y(r_{min})= -\frac{T}{2n-d}\left(\frac{T}{n}\frac{n-d}{2n-d}\right)^{\frac{n-d}{n}}\ .
\ee
Now it is easy to see that there exist precisely two solutions 
\be
0 < r_{-} < r_{min} < r_{+} < \left(\frac{T}{n}\right)^{\frac{d}{2n}}
\ee
for $0<|t|<t_{max}$ where
\be
t_{max}=\frac{n^2}{d}\left(\frac{T}{n}\frac{n-d}{2n-d}\right)^{\frac{n-d}{2n}}\sqrt{\frac{T}{n(n-d)(2n-d)}}\ .
\ee
Since
\be
\frac{t_{max}}{t_{cr}}=\left(\frac{n}{n-d}\right)^{\frac{d}{2n}}<1\ ,
\ee
this condition is always satisfied for $0<|t|<t_{cr}$. Moreover,
\be
\frac{r_{cr}}{r_{min}} =\left(\frac{n}{n-d}\right)^{\frac{d}{2n}}>1\ .
\ee
Therefore with the choice $r=r_{+}$ we have
\be
|\alpha|^{2}=\frac{n}{n-d}\left(\frac{T}{n}r_{+}^{-\frac{2n}{d}}-1\right)<\frac{n}{n-d}\left(\frac{T}{n}r_{cr}^{-\frac{2n}{d}}-1\right)=\frac{n}{n-d}\left(\frac{2n-d}{n}-1\right)=1 \ .
\ee
In order to prove that $r_{-}$ does not satisfy \eqref{ineq_a1} we write
\be
|\alpha|^{2}=\frac{n}{n-d}\left(\frac{T}{n}r_{-}^{-\frac{2n}{d}}-1\right)>\frac{n}{n-d}\left(\frac{T}{n}r_{min}^{-\frac{2n}{d}}-1\right)=\frac{n}{n-d}\left(\frac{2n-d}{n-d}-1\right)=\left(\frac{n}{n-d}\right)^2>1 \ .
\ee
Therefore the only positive solution compatible with \eqref{ineq_a} is $r=r_{+}(t)$.

For $n< d< 2n$ the function $y(r)$ is strictly increasing with
\be
y(r)\to -\infty \quad r\to 0_{+} \quad \mbox{ and }\quad y(r) \to \infty \quad r \to \infty\ ,
\ee
and therefore \eqref{eq:r_eq} has a unique solution for every value of $|t|$.
Since the unique root of $y(r)=0$ is given by
\be
r=\left(\frac{T}{n}\right)^{\frac{d}{2n}}\ ,
\ee
for $0<|t|<t_{cr}$ we have 
\be
\left(\frac{T}{n}\right)^{\frac{d}{2n}}<r_{0}(|t|)<r_{cr}
\ee
and hence
\be
|\alpha|^{2}=\frac{n}{d-n}\left(1-\frac{T}{n}r_{0}^{-\frac{2n}{d}}\right)<\frac{n}{d-n}\left(1-\frac{T}{n}r_{cr}^{-\frac{2n}{d}}\right)=\frac{n}{d-n}\left(1-\frac{2n-d}{n}\right)=1\ ,
\ee
as needed.
\end{proof}

 \bibliographystyle{abbrv}
 \bibliography{eq_op_dihedral}
\end{document}